\documentclass[a4paper]{article}

\makeatletter

\newcommand{\Rmnum}[1]{\expandafter\@slowromancap\romannumeral #1@}
\makeatother
\usepackage{graphicx}
\usepackage{amsfonts}
\usepackage{amsmath}
\usepackage{amssymb}
\usepackage{fancyhdr}
\usepackage{titlesec}
\usepackage{indentfirst}
\usepackage{tikz}
\usepackage{tikz-cd}
\usepackage{booktabs}
\usepackage{verbatim}
\usepackage{color}
\usepackage{amsthm}
\usepackage{subfigure}
\usepackage{abstract}
\usepackage{multirow}
\usepackage{fullpage}
\usepackage{hyperref}
\usepackage[margin=0pt]{geometry}
\usepackage{lipsum}
\usepackage{layout}
\usepackage{mathtools}
\usepackage[page,header]{appendix}
\usepackage{titletoc}
\usepackage{newfloat}
\usepackage{csquotes}
\numberwithin{equation}{section}
\newtheorem{theorem}{Theorem}[section]
\newtheorem{lemma}[theorem]{Lemma}

\newtheorem{corollary}[theorem]{Corollary}

\newtheorem{definition}[theorem]{Definition}
\newtheorem{remark}[theorem]{Remark}

\newtheorem{conjecture}{Conjecture}

\newcommand{\RNum}[1]{\uppercase\expandafter{\romannumeral #1\relax}}

\DeclareMathOperator{\disc}{disc}
\DeclareMathOperator{\Gal}{Gal}

\DeclareMathOperator{\Disc}{Disc}
\DeclareMathOperator{\Nm}{Nm}

\DeclareMathOperator{\Ker}{Ker}

\DeclareMathOperator{\Cl}{Cl}

\DeclareMathOperator{\Cond}{Cond}

\DeclareMathOperator{\Frob}{Frob}

\usepackage[OT2,T1]{fontenc}
\DeclareSymbolFont{cyrletters}{OT2}{wncyr}{m}{n}
\DeclareMathSymbol{\Sha}{\mathalpha}{cyrletters}{"58}

\newcommand{\zp}{\mathbb{Z}/p\mathbb{Z}}

\newcommand{\C}{\ensuremath{{\mathbb{C}}}}

\newcommand{\Q}{\ensuremath{{\mathbb{Q}}}}

\DeclareFloatingEnvironment[fileext=dia,within=section]{diagram}

\usepackage{pdfsync}

\begin{document}

\title{ Pointwise Bound for $\ell$-torsion in Class Groups: Elementary Abelian Extensions}
\author{Jiuya Wang}
\newcommand{\Addresses}{
	{
				
		\bigskip
		\footnotesize		
		Jiuya Wang, \textsc{Department of Mathematics, Duke University, 120 Science Drive 117 Physics Building Durham, NC 27708, USA
		}\par\nopagebreak
		\textit{E-mail address}: \texttt{wangjiuy@math.duke.edu}	
	}}
\maketitle	
	\begin{abstract}	
       Elementary abelian groups are finite groups in the form of $A=(\zp)^r$ for a prime number $p$. For every integer $\ell>1$ and $r>1$, we prove a non-trivial upper bound on the $\ell$-torsion in class groups of every $A$-extension. Our results are pointwise and unconditional. When $r$ is large enough, the pointwise bound we obtain also breaks the previously best known bound shown by Ellenberg-Venkatesh under GRH.
	\end{abstract}
\pagenumbering{arabic}	
\section{Introduction}
In this paper, we study cases of the following conjecture. 
\begin{conjecture}[$\ell$-torsion Conjecture]\label{conj:l-torsion}
	Given an integer $\ell>1$ and a number field $k$. For any degree $d$ extension $F/k$, the size of $\ell$-torsion in the class group of $F$ is bounded by
	$$|\Cl_{F}[\ell]| \le O_{\epsilon,k}(\Disc(F)^{\epsilon}). $$
\end{conjecture}

This conjecture has been brought forward previously by \cite{BruSil,Duk98,Zha05}. Nowadays in arithmetic statistics, Conjecture \ref{conj:l-torsion} has been closely related to other questions. In \cite{PTBW2}, it is shown that Conjecture \ref{conj:l-torsion} is implied by assuming a moment version of Cohen-Lenstra heuristics. Conjecture \ref{conj:l-torsion} is also closely related to proving upper bounds in counting number fields \cite{Klu,KluDp,Klu12,Widmer1,Alb,ML17}, number of ellliptic curves with a fixed conductor \cite{BruSil}, number of integral points of elliptic curves, and size of Selmer groups and ranks of elliptic curves and hyperelliptic curves \cite{BruKra,BSTTTZ,HV06}. 

By a theorem of Brauer-Siegel, see for example \cite{SL}, the class number of $F$ with $[F:\mathbb{Q}]=d$ is bounded by $O_{\epsilon,d}(\Disc(F)^{1/2+\epsilon})$, therefore we get the so-called \emph{trivial bound} for $\ell$-torsion in class groups:
\begin{equation}\label{eqn:trivial-bound}
|\Cl_F[\ell]|  \le O_{\epsilon}(\Disc(F)^{1/2+\epsilon}).
\end{equation}
As one can observe, there is a huge gap between the trivial bound and Conjecture \ref{conj:l-torsion}. The only case where Conjecture \ref{conj:l-torsion} is proved to the full strength is when $(d,\ell)=(2,2)$ due to Gauss by genus theory. Aside from this special case, it is even wildly open to prove a result in the form of (\ref{eqn:trivial-bound}) by replacing $1/2$ with any $0<1/2-\delta<1/2$. We will call such a bound a \emph{non-trivial bound} for $\ell$-torsion in class groups. Notice that for a fixed degree, there are only finitely many possible Galois groups, and fields with different Galois groups behave very differently. Therefore it is natural to split up discussions of $(d, \ell)$ to $(G, \ell)$ for a transitive permutation group $G\subset S_d$ with degree $d$, that is, considering the bound for $\Cl_F[\ell]$ where the Galois closure $\tilde{F}$ of $F/k$ has $\Gal(\tilde{F}/k)=G$, see works towards this question \cite{ML17,Widmer1,FreWid18,Chen,FreWid18x}. Aside from special cases that can be handled by genus theory, previously people can only get non-trivial bound for $(G, \ell)$: when $\ell = 2$ for all Galois groups $G$ (i.e., for all degree $d$), see \cite{BSTTTZ}, and $\ell=3$ for all small degree number fields with $d\le 4$, see \cite{EV07,Pie05,HV06}. In terms of conditional results, the work of Ellenberg-Venkatesh \cite{EV07} shows a non-trivial bound for all $G$ and all $\ell$ in the order of $O_{\epsilon,k}(\Disc(F)^{1/2-\frac{1}{2\ell(d-1)}+\epsilon})$ where $d= [F:k]$ by assuming GRH. Indeed, a critical lemma in \cite{EV07} shows that $|\Cl_F[\ell]|$ can be non-trivially bounded as long as there exist many small split primes, which is guaranteed by GRH in general. See Lemma \ref{lem:EVW} for a precise statement. Recently there has been an emerging group of works, see e.g. \cite{Ellen16,ML17,Widmer1,FreWid18,Chen,FreWid18x,ZTArtin}, towards removing the GRH condition in \cite{EV07}. All of these works only obtain results \emph{on average} in order to remove GRH. More precisely, such average results prove that a non-trivial bound holds for number fields within a family of number fields with a possible zero-density exceptional set. 

In this paper, we will focus on cases where $G$ is an elementary abelian group with rank $r>1$ and $\ell>1$. In particular, we obtain a genuinely pointwise bound on $|\Cl_F[\ell]|$ for arbitrary $\ell>1$ that is unconditional. We prove the following theorem.
\begin{theorem}[Theorem \ref{thm:odd-final-results-overk}, Theorem \ref{thm:even-induction-overQ} and Theorem \ref{thm:even-induction-overk} ]\label{thm:main}
Given $A= (\zp)^r$ where $r>1$ and an integer $\ell>1$. There exists $\delta(\ell,p)>0$ such that for any $A$-extension $L/\Q$, we have
$$|\Cl_{L}[\ell]| \le O_{\epsilon}( \Disc(L)^{1/2-\delta(\ell,p)+\epsilon}).$$
\end{theorem}
\begin{remark}
Analogues of Theorem \ref{thm:main} over general number field $k$ are also proved, see Theorem \ref{thm:odd-final-results-overk} and Theorem \ref{thm:even-induction-overk}, where different savings $\delta_k(\ell, p)$ are obtained also depending on $k$. Here in order to state a uniform result in Theorem \ref{thm:main}, for $p$ odd, the saving $\delta(\ell,p)$ is taken to be $\delta_{\Q}(\ell_{(p)}, p)$ in Theorem \ref{thm:odd-final-results-overk}; and for $p=2$, the saving $\delta(\ell, 2)$ is taken to be $\delta_{\Q}(\ell_{(2)}, 2)$ in Theorem \ref{thm:even-induction-overk} and \ref{thm:even-rank2-results-overk}. For $p=2$ and $r>2$, a better saving is stated in Theorem \ref{thm:even-induction-overQ}. For $\ell=p$, of course we have a much better bound $|\Cl_L[p]| \le O_{\epsilon} (\Disc(L)^{\epsilon})$ by genus theory, for example see Theorem $3$ in \cite{Cor83}. All results in this paper are effective. 
\end{remark}

It is worth noticing that this is the first family of Galois groups $G$ where $\ell$-torsion in class groups of $G$-extensions are bounded non-trivially for every integer $\ell>1$ unconditionally. 

A very important characteristic of the savings $\delta(\ell,p)$ in Theorem \ref{thm:main} (including its analogue $\delta_k(\ell,p)$ over general number field $k$) is that it does not depend on the rank $r$ of $A$. Therefore there exists $r_0 = r_0( \ell, p)$ such that when $r> r_0(\ell,p)$, we have
$$\delta(\ell, p) \ge \frac{1}{2\ell(d-1)} = \frac{1}{2\ell(p^r-1)} = \text{ the saving proved in \cite{EV07} by assuming GRH}.$$

The main strategy of this work is to take advantage of the group structure of elementary abelian groups $A = (\zp)^r$ with $r>1$.
\begin{itemize}
	\item Firstly, an elementary abelian group $A = (\zp)^r$ has $(p^r-1)/(p-1)$ index-$p$ subgroups $A_i$ with $A/A_i\simeq \zp$. Correspondingly, for any $A$-extension $L/k$ with $\Gal(L/k)=A$, we get $(p^r-1)/(p-1)$ degree $p$ sub-extensions $K_i/k$ with Galois group $\Gal(K_i/k)=\zp$. Considering $\Cl_{L/k}[\ell]$ as a Galois module with Galois group $\Gal(L/k)$, it can be decomposed along the fixed part by $A_i$ for all index-$p$ subgroup $A_i$, therefore we can obtain equalities like 
	$$|\Cl_{L/k}[\ell]| = \prod_{K_i/k} |\Cl_{K_i/k}[\ell]|,\quad\quad \Disc(L/k)= \prod_{K_i/k} \Disc(K_i/k),$$
	where $K_i/k$ ranges over all degree $p$ subfields of $L$. For more details on these equalities, see Lemma \ref{lem:class-grp-decomposition} and \ref{lem:disc-prod}. Therefore we can reduce the question of $L/k$ to the question of subfields $K_i/k$. This is essentially the key reason why the bound we obtain behave better than the GRH bound when $r$ is sufficiently large.
	\item
	Secondly, the decomposition group of an $A$-extension at unramified primes must be $\zp$ since every cyclic subgroup of $A$ is isomorphic to $\zp$. Therefore every unramified prime $p$ is at least split in $(p^{r-1}-1)/(p-1)$ degree $p$ subfields of $L/k$. This guarantees the existence of split primes. 
	\item
	Thirdly, by the conductor-discriminant formula, we can give lower bound on the discriminant of subfields, see for example Lemma \ref{lem:rank2-disc-bound}. Then we can apply results on upper or lower bounds of prime counting functions where the range of consideration is in the order of a polynomial in the modulus, see section \ref{sec:analytic} for a collection of some results in this direction that we use, and see Theorem \ref{thm:odd-incomparable-overQ} for an example how we apply them. 
\end{itemize}

The organization of the paper is as follows. In section \ref{sec:algebraic}, we introduce the algebraic lemmas in preparation for the later proof. It includes several necessary equalities of class groups and discriminants and inequalities of discriminants. In section \ref{sec:analytic}, we collect several results on upper and lower bounds of prime counting function. They all share the property that the range of primes considered is in a polynomial order of the discriminant. In section \ref{sec:E-V}, we revisit the critical lemma from \cite{EV07} on bounding $\ell$-torsion in the class groups conditional on the existence of small split primes. In section \ref{sec:odd}, we give the proof of Theorem \ref{thm:main}, including its analogue over general number field, when $p$ is odd. In section \ref{sec:even}, we give the proof of Theorem \ref{thm:main}, including its analogue over general number field, when $p=2$. We mention that section \ref{sec:odd} and \ref{sec:even} share a lot of similarities in spirit, whereas section \ref{sec:even} deals with some new complication when $p=2$. In order to grasp the main idea, it is recommended to read section \ref{sec:odd} first. 

\section{Notations}
\noindent
$k$: a number field considered as the base field\\
$|\cdot|$: the absolute norm $\Nm_{k/\Q}$\\
$\Gal(F/k)$: Galois group of $F/k$\\
$\Disc(F/k)$: relative discriminant $|\disc(F/k)|$ of $F/k$ where $\disc(F/k)$ is the relative discriminant ideal in $k$, when $k=\mathbb{Q}$ it is the usual absolute discriminant\\
$\Cl_{F/k}$: relative class group of $F/k$, when $k=\mathbb{Q}$ it is the usual class group of $F$\\
$\Cl_{F/k}[\ell]$:  $\{ [\alpha]\in \Cl_{F/k}\mid \ell [\alpha] = 0\in \Cl_{F/k} \}$\\
$|\Cl_{F/k}[\ell]|$, $|\Cl_{F}[\ell]|$: the size of $\Cl_{F/k}[\ell]$, $\Cl_{F}[\ell]$\\
$M^G$: the maximal submodule of the $G$-module $M$ that is invariant under $G$\\
$M_G$: the maximal quotient module $M/I_G(M)$ of the $G$-module $M$ that is invariant under $G$\\
$I_G$: the augmentation ideal $\large \langle \sigma-1 \mid \sigma\in G \large\rangle \subset R[G]$ in the group ring with coefficient ring $R$\\
$\pi(Y; q, a)$: the number of prime numbers $p$ such that $p< Y$ and $p\equiv a \mod q$\\
$\pi(Y; L/k, \mathcal{C})$: the number of unramified prime ideals $p$ in $k$ with $|p|<Y$ and $\Frob_{p} \in \mathcal{C}$ where $\mathcal{C}$ is a conjugacy class of $\Gal(L/k)$ \\
$\pi(Y; L/k, \hat{C})$: the number of unramified prime ideals $p$ in $L$ with $|p|<Y$ and $\Frob_{p} \notin \mathcal{C}$ where $\mathcal{C}$ is a conjugacy class of $\Gal(L/k)$\\
$A\asymp B$: there exist absolute constants $C_1$ and $C_2$ such that $C_1 B\le A\le C_2 B$\\
$\Delta(\ell,d)$: a constant number slightly smaller than $\frac{1}{2\ell(d-1)}$, see Remark \ref{rmk:Delta}\\
$\ell_{(p)}$: the maximal factor of $\ell$ that is relatively prime to a prime number $p$ for an integer $\ell>1$\\
$\eta(L/k)$: see (\ref{def:eta-odd}) when $\Gal(L/k) = \zp\times \zp$ with $p$ odd and see (\ref{def:eta-rank3}) when $\Gal(L/k) = (\mathbb{Z}/2\mathbb{Z})^3$ \\
$\eta_0(\ell, p)_k$: a cut-off for $\eta(L/k)$ that is determined in Theorem \ref{thm:odd-incomparable-overQ} and \ref{thm:odd-incomparable-overk} when $\Gal(L/k)$ has rank $2$; we will drop $k$ when $k=\Q$\\
$\delta$: through out the paper we always use $\delta$ to denote a power saving from the trivial power $1/2$ in the bound; we use $\delta_c$ to denote the power saving when $\eta(L/k)$ is small and $\delta_{ic}$ to denote the power saving when $\eta(L/k)$ is big. \emph{Small} and \emph{big} are quantified by comparing to $\eta_0(\ell,p)_k$. \\

Warning: In order to simplify the notation for the whole paper, unless specifically mentioned otherwise, the implied constants $O_{\epsilon}$, $O_{\epsilon,k}$, $O_{\epsilon, k, \epsilon_0}$ will always depend on $\ell, d$ aside from the dependence indicated in the symbol when we are stating results or conjectures on bounding $\ell$-torsion in class groups of degree $d$ extensions.\\

\section{Algebraic Theory}\label{sec:algebraic}
In this section, firstly we are going to state several standard equalities of class group and discriminants, Lemma \ref{lem:class-grp-decomposition} and \ref{lem:disc-prod} from algebraic number theory that will be of crucial use for later proof. These results and equalities are known previously, for example see \cite{CM87}. Here we only include a proof for the convenience of the readers. Secondly, we will give a ramification analysis on $A$-extensions and prove critical lemmas Lemma \ref{lem:rank2-disc-bound} and \ref{lem:rank3-disc-bound} throughout the proof. 

\subsection{Relative Class Group}\label{ssec:relative-class-grp}
In this section, we define the notion of relative class group. The relative class group $\Cl_{F/k} \subset \Cl_F$ is defined to be $\Ker(\Nm)$ where $\Nm: \Cl_F\to \Cl_k$ is induced from the usual norm on fractional ideals of $F$. 

Fix an integer $\ell>1$ that is relatively prime to the degree $[F:k]$, we will show that the following forms a short exact sequence
$$0 \to \Cl_{F/k}[\ell] \to \Cl_F[\ell] \to \Cl_k[\ell] \to 0. $$
Indeed, denote the map $\iota: \Cl_k\to \Cl_F$ which is induced from the usual embedding of fractional ideals. We know that $\Nm\circ \iota: \Cl_k \to \Cl_k$ is equivalent to multiplication by $[F:k]$, which is an isomorphism on the $\ell$-torsion part $\Cl_{k}[\ell]$. Therefore $\Nm: \Cl_F[\ell] \to \Cl_k[\ell]$ is surjective and $\iota: \Cl_k[\ell] \to \Cl_F[\ell]$ is injective and gives a section of the short exact sequence above.

If $F/k$ is Galois with $\Gal(F/k)=G$, then the class group $\Cl_F[\ell]$ can be considered as a Galois module with Galois group $G$. Since $(|G|,\ell)=1$, the Tate cohomology $\hat{H}^i(G, \Cl_F[\ell])$ vanishes for every $i$. It follows from $\hat{H}^0(G, \Cl_F[\ell]) =  (\Cl_F[\ell])^G/\iota\circ \Nm(\Cl_F[\ell])=0$ that $(\Cl_F[\ell])^G =\iota\circ \Nm(\Cl_F[\ell]) = \iota(\Cl_k[\ell]) \simeq \Cl_k[\ell]$. The last two equalities come from $\Nm$ being surjective and $\iota$ being injective. Similarly, it follows from $\hat{H}^{-1}(G, \Cl_F[\ell]) =  \Cl_{F/k}[\ell]/I_G (\Cl_F[\ell])=0$ that $(\Cl_F[\ell])_G = \Cl_F[\ell]/I_G(\Cl_F[\ell]) = \Cl_F[\ell]/\Cl_{F/k}[\ell]\simeq \Cl_k[\ell]$.

\subsection{Class Group Decomposition}
The main goal of the following lemma is to reduce the questions about elementary abelian extensions to those of their sub-extensions. 
\begin{lemma}\label{lem:class-grp-decomposition}
	Given an elementary abelian group $A =( \mathbb{Z}/p\mathbb{Z})^r$ with $r>1$ and an integer $\ell>1$ with $(\ell, p) =1$. For any $A$-extension $L/k$, 
	$$|\Cl_{L/k}[\ell]| = \prod_{K_i/k} |\Cl_{K_i/k}[\ell]|,$$
	where $K_i/k$ ranges over all subfields of $L$ with $[K_i:k]=p$. 
\end{lemma}
\begin{proof}
The class group $\Cl_{L/k}[\ell]$ is naturally an $\mathbb{Z}/\ell\mathbb{Z}[A]$-module since $\Gal(L/k)$ acts on it.  For an elementary group $A$ and an integer $\ell$ with $(|A|, \ell)=1$, we have that $\mathbb{Z}/\ell\mathbb{Z}[A]$ is semi-simple by Maschke's theorem. We can decompose the augmentation ideal
$$I_A = \oplus_i \epsilon_i I_A,$$
where $\epsilon_i = \frac{1}{|A_i|} \sum_{a\in A_i} a$ and $A_i$ ranges over all index-$p$ subgroup of $A$. It can be easily shown that $\epsilon_i^2= \epsilon_i$ and $\epsilon_i \circ \epsilon_j I_A = 0$. Therefore any faithful $\mathbb{Z}/\ell\mathbb{Z}[A]$-module $M$ (meaning $M_A$ is trivial), $M$ can be written as a direct sum
$$M = M\otimes (\mathbb{Z}/\ell\mathbb{Z})[A]= M \otimes I_A \oplus M\otimes (\mathbb{Z}/\ell\mathbb{Z})[A]/I_A=  \oplus_i \epsilon_i M \oplus M_A =\oplus_i \epsilon_i M,$$
where the summation is over all index-$p$ subgroups $A_i\subset A$. 

By the discussion in section \ref{ssec:relative-class-grp}, the module $M = \Cl_{L/k}[\ell]$ as a submodule of $\Cl_L[\ell]$ is faithful: it can be easily seen by applying $(\cdot )_G$ to the short exact sequence in section \ref{ssec:relative-class-grp} and noticing $\Cl_F[\ell]_G \simeq \Cl_k[\ell]$. Given $\epsilon_i$ corresponding to $A_i\subset A$ and $K_i$ the field fixed by $A_i$, the sub-module $\epsilon_i M = \Nm_{A_i}(M) = \Cl_{L/k}[\ell]/ \Cl_{L/K_i}[\ell]$: it can be seen by the following diagram.  Therefore $|\epsilon_i M| = |\Cl_{K_i/k}[\ell]|$.
\begin{center}\label{diag:relative-class-grp}
	\begin{tikzcd}				
		0\arrow{r} & \Cl_{L/K_i}[\ell] \cap \Cl_{L/k}[\ell] = \Cl_{L/K_i}[\ell] \arrow{r}\arrow[d, hook] & \Cl_{L/k}[\ell] \arrow{r}\arrow[d, hook] & \Nm_{A_i}(\Cl_{L/k}[\ell]) \arrow{r}\arrow[d,hook] & 0 \\		
		0\arrow{r} & \Cl_{L/K_i}[\ell]\arrow{r} & \Cl_L[\ell] \arrow[r,"\Nm_{A_i}"] & \Cl_{K_i}[\ell] \arrow{r} & 0 \\
	\end{tikzcd}
\end{center}
\end{proof}
Next we apply Lemma \ref{lem:class-grp-decomposition} to degree $p^2$ subfields of $A= (\zp)^r$ with $r>2$. Notice that every $K_i$ are contained in exactly $(p^{r-1}-1)/(p-1)$ subfields $M_j$ with $[M_j:k]=p^2$, so we have the following equality. 
\begin{corollary}\label{coro:class-grp-second-layer-decomposition}
	Given an elementary abelian group $A = (\zp)^r$ with $r>2$ and an integer $\ell>1$ with $(\ell, p) =1$. For any $A$-extension $L/k$, 
	$$|\Cl_{L/k}[\ell]| = \prod_{M_j/k} |\Cl_{M_j/k} [\ell]|^{(p-1)/(p^{r-1}-1)} = \prod_{F_s/k} |\Cl_{F_s/k}[\ell]|^{(p-1)/(p^{r+1-t}-1)},$$
	where $M_j/k$ ranges over all subfields of $L$ with $[M_j:k]=p^2$, and $F_s/k$ ranges over all subfields of $L$ with $[F_s:k] = p^t$. 
\end{corollary}

\subsection{Ramification Analysis}
The main goal of this section is to give an analysis on the discriminants of all sub-extensions of $L/k$ when $\Gal(L/k)=A$ and $A$ is an elementary abelian group. 

\begin{lemma}\label{lem:disc-prod}
	Given an elementary group $A =( \mathbb{Z}/p\mathbb{Z})^r$ with $r>1$. For any $A$-extension $L/k$, we have
	$$\Disc(L/k)= \prod_{K_i/k} \Disc(K_i/k),$$
	where $K_i/k$ ranges over all subfield of $L$ with $[K_i:k]=p$. 
\end{lemma}
\begin{proof}
Recall that $\Disc(L/k)$ is the Artin-conductor of $L/k$ with the representation$\rho$ of $A$ where $\rho$ is the regular representation of $A$ over $\C$. Then $\rho-1 = \oplus_i \rho_i$ where $\rho_i = (\rho-1)^{A_i} =\rho^{A_i} - 1$ where $1$ is denoted to be the trivial representation of $A$. Therefore notice that the Artin-conductor with trivial character is trivial, we get the Artin conductor $\mathfrak{f}$ associated to $\rho$ is decomposed as:
$$\Disc(L/k) = \mathfrak{f}_{L/k}(\rho) = \prod_{[A:A_i=p]} \mathfrak{f}_{L/k}(\rho^{A_i}) = \prod_{[K_i:k]=p} \mathfrak{f}_{K_i/k}(\rho_i) = \prod_{[K_i:k]=p}\Disc(K_i/k).$$
\end{proof}
Similarly with Corollary \ref{coro:class-grp-second-layer-decomposition}, we also have
\begin{equation}\label{eqn:disc-prod-second-layer}
\Disc(L/k)= \prod_{M_j/k} \Disc(M_j/k)^{(p-1)/(p^{r-1}-1)} = \prod_{F_s/k} |\Disc(F_s/k)|^{(p-1)/(p^{r+1-t}-1)},
\end{equation}
where $M_j/k$ ranges over all subfields of $L$ with $[M_j:k]=p^2$ and $F_s/k$ ranges over all subfields of $L$ with $[F_s:k]= p^t$. 

\begin{lemma}\label{lem:rank2-disc-bound}
	Given an elementary group $A =( \mathbb{Z}/p\mathbb{Z})^2$. For any $A$-extension $L/k$, denote $K_1$ and $K_2$ to be two arbitrary subfields of $L/k$ with degree $p$. Given $\eta = \frac{\ln \Disc(K_2/k)}{ \ln \Disc(K_1/k)}$. Then we have a lower bound for $\Disc(K_1/k)$ and $\Disc(K_2/k)$ as following
	$$ \Disc(K_1/k) \ge \Disc(L/k)^{1/p(\eta+1)}, \quad\quad \Disc(K_2/k) \ge \Disc(L/k)^{\eta/p(\eta+1)}.$$
\end{lemma}
\begin{proof}
	By the conductor discriminant formula, we have that the discriminant of the compositum satisfies the following inequality, see for example \cite[Theorem $2.1$]{JW17}
	$$\Disc(K_1/k)^p\cdot \Disc(K_2/k)^p \ge \Disc(L/k).$$
	By assumption, we have
	$$\Disc(K_1/k)^{p(\eta+1)} \ge \Disc(L/k),$$
	therefore
	$$\Disc(K_1/k) \ge \Disc(L/k)^{1/p(\eta+1)},\quad\quad \Disc(K_2/k) \ge \Disc(L/k)^{\eta/p(\eta+1)}.$$	
\end{proof}

A similar proof yields the following lower bound for $A = (\mathbb{Z}/2\mathbb{Z})^3$. We will need to use the following lemma when we discuss the abelian group $A = (\mathbb{Z}/2\mathbb{Z})^3$ in section \ref{ssec:even-rank3-comparable} and \ref{ssec:even-rank3-incomparable}.  

\begin{lemma}\label{lem:rank3-disc-bound}
	Given the elementary abelian group $A = (\mathbb{Z}/2\mathbb{Z})^3$. For any $A$-extension $L/k$, denote $M/k$ to be a quartic subfield of $L/k$ and $K/k$ to be a quadratic subfield of $L/k$ that is not a quadratic subfield of $M/k$. Given $\eta = \frac{\ln \Disc(K/k)}{ \ln \Disc(M/k)}$, we have
	$$\Disc(M/k) \ge \Disc(L/k)^{1/(4\eta+2)}, \quad\quad \Disc(K/k) \ge \Disc(L/k)^{\eta/(4\eta+2)}.$$
\end{lemma}

\section{Analytic Theory}\label{sec:analytic}
As a preparation for the main proof, we are going to state Brun-Titchmarsh theorem \cite{MontVau} and a lower bound theorem in \cite{May}, and generalizations of \cite{May} to general number fields \cite{ZamThesis} that we can conveniently use. Results in this direction have also appeared previously in \cite{Weiss, Deb,ZTleast,ZTBrunT}. We apply the following statements in our proofs since the format of the statements is convenient to use in our application. 

The main reason that these bounds are good for us is that they hold for $x> f(q)$ where $x$ is the range of consideration, $q$ is the modulus and $f(q)$ is some polynomial in $q$. 

\begin{lemma}[Brun-Titchmarsh, \cite{MontVau}]\label{lem:B-T}
For $x>q$, we have
$$ \pi(x; q, a) \le \frac{2}{1-\ln q/\ln x}\cdot \frac{x}{\phi(q) \ln x}.$$
\end{lemma}

\begin{lemma}[\cite{May},Theorem $3.2$]\label{lem:Maynard}

	For $x\ge q^8$, there exists an absolute constant $C>0$ and an effectively computable constant $q_2$ such that for $q\ge q_2$, we have
	$$ \pi(x; q, a) \ge C\frac{\ln q}{q^{1/2}} \cdot \frac{x}{\phi(q)\ln x}.$$
\end{lemma}

\begin{lemma}[\cite{ZamThesis}, Theorem $1.3.1$ \cite{ZTBrunT}, Theorem $1.2$]\label{lem:ZT}
Given $L/k$ a Galois extension of number fields with $[L:\Q]=d$. There exists absolute, effective constants $\gamma = \gamma(k, G)>2$, $\beta = \beta(k, G)>2$, $D_0 = D_0(k)>0$ and $C=C(k)>0$ such that if $\Disc(L/k)\ge D_0$, then for $x\ge \Disc(L/k)^{\beta}$, we have
$$C_k\frac{1}{\Disc(L/k)^{\gamma}}\cdot \frac{|\mathcal{C}|}{|G|} \cdot \frac{x}{\ln x} \le \pi(x;L/k, \mathcal{C}) \le (2+ O(d x^{-\frac{1}{166d+327}} ))\cdot  \frac{|\mathcal{C}|}{|G|}\cdot \frac{x}{\ln x}.$$
\end{lemma}

We will navigate where these theorems are used in this paper. For results over $\Q$, we use Lemma \ref{lem:B-T} in section \ref{sec:odd} for all odd degree extensions, in section \ref{sec:even} for all even degree extensions with rank $r>2$; we use Lemma \ref{lem:Maynard} in section \ref{ssec:even-rank3-incomparable} for $(\mathbb{Z}/2\mathbb{Z})^2$ extensions. For results over general number field $k$, we did not seek after an optimal bound in this work. For simplicity, we always use the lower bound in Lemma \ref{lem:ZT}, see both section \ref{sec:odd} and \ref{sec:even}. The main reason for doing this is that by using the lower bound, we can write down the power saving away from the trivial bound explicitly in terms of $\beta(k, G)$ and $\gamma(k, G)$. And these numbers are determined explicitly in previous work: for example, in Theorem $1.3.1$ in \cite{ZamThesis}, if we only consider the lower bound side, then $\gamma(k, G)$ can be taken to be $19$ and $\beta(k, G)$ can be taken to be $35$. The upper bound in Lemma \ref{lem:ZT} can also be used to obtain a non-trivial bound for $(\zp)^r$-extensions over $k$ with $r>1$, following a similar proof over $\Q$ in Theorem \ref{thm:odd-incomparable-overQ}. However we did not use them in this paper since the saving will depend on the implied constant in the error term $O(dx^{-1/(166d+327)})$. 

\section{Ellenberg-Venkatesh Revisited}\label{sec:E-V}
In this section, we will revisit \cite{EV07} and rephrase their critical lemma that we base on.  By defining the notion of $\Delta$-good/bad in Definition \ref{bad}, we rephrase this lemma in Lemma \ref{lem:EVW} in the form that we can conveniently use. 

Given an element $a\in A$ in an abelian group $A$ (or a conjugacy class $\mathcal{C}\subset G$ for a general finite group $G$), for a Galois extension $L/k$, we denote $\pi(Y;L/k, a)$ (or $\pi(Y;L/k, \mathcal{C})$) to be the number of unramified primes ideals $p$ in $k$ with $\Frob_{p} = a \in A$ (or $\Frob_{p} \in \mathcal{C} \subset G$). We will always denote $e\in A$ (or $e\in G$) to mean the identity element, and $\Frob_{p} = e\in A$ (or $\Frob_{p} = e\in G$) corresponds to $p$ splitting in $L/k$. We will denote $\pi(Y; L/k, \hat{a})$ to be the number of primes ideals $p$ in $k$ with $\Frob_{p} \in A\backslash \{ a\}$. 

We define
\begin{equation}\label{eqn:definition-bad}
\mathcal{B}(G, \theta, c) := \Big\{ L/k \mid  \Gal(L/k) = G, \pi(\Disc(L/k)^{\theta}; L/k, e) \le c\frac{\Disc(L/k)^{\theta}}{\ln \Disc(L/k)^{\theta}}\Big \},
\end{equation}
where $c>0$ is an absolute small number. In reality, the choice of $c$ will be determined from the proof. 

\begin{definition}\label{bad}
	Given $\Delta>0$, we call an extension $L/k$ \emph{$\Delta$-bad with respect to $c$ } if $L/k\in \mathcal{B}(A, \Delta, c)$ where $A = \Gal(L/k)$. If $L/k$ is not $\Delta$-bad with respect to $c$, we will say $L/k$ is \emph{$\Delta$-good with respect to $c$}. When $c$ is clear in the set up, we will simply say \emph{$\Delta$-bad} or \emph{$\Delta$-good}.
\end{definition}
The following is the critical lemma from \cite{EV07}. 
\begin{lemma}[\cite{EV07}]\label{EVS}
	Given a Galois extension $L/k$ and $0< \theta< \frac{1}{2\ell(d-1)}$, denote $$M:=\pi(\Disc(L/k)^{\theta};L/k, e),$$ then 
	\begin{equation}\label{eqn:EV-lemma}
	|\Cl_{L}[\ell]|\le O_{\epsilon,k}\Big(\frac{\Disc(L)^{1/2+\epsilon}}{M}\Big).
	\end{equation}
\end{lemma}
\begin{remark}[Transition between Absolute/Relative setting]\label{rmk:abs-to-rel}
When $(\ell, [L:k]) = 1$, we have $|\Cl_{L}[\ell]| = |\Cl_{L/k}[\ell]|\cdot |\Cl_k[\ell]|$. Notice that we always have $\Disc(L) = \Disc(L/k) \cdot \Disc(k)^{[L:k]}$, we can easily adapt the original statement (\ref{eqn:EV-lemma}) to the statement about $\Cl_{L/k}$ and $\Disc(L/k)$:
	\begin{equation}\label{eqn:EV-lemma-2}
		|\Cl_{L/k}[\ell]|\le O_{\epsilon,k}\Big(\frac{\Disc(L/k)^{1/2+\epsilon}}{M}\Big).
	\end{equation}

More specifically, fix a number field $k$, an elementary abelian group $A$ and an integer $\ell>1$ with $(\ell, |A|)=1$. Denote $\mathcal{F}$ to be the set of all $L/k$ with $\Gal(L/k)=A$, then
\begin{equation}\label{eqn:abs-to-rel}
\begin{aligned}
&\exists \delta>0, \forall L/k \in \mathcal{F}, &|\Cl_{L/k}[\ell]| &\le O_{k, \epsilon} (\Disc(L/k)^{1/2-\delta+\epsilon}) \iff \\
& \exists \delta>0 ,\forall L/k \in \mathcal{F},& |\Cl_{L}[\ell]| &\le O_{k, \epsilon} (\Disc(L)^{1/2-\delta+\epsilon}).
 \end{aligned}
\end{equation}
Since the two statements are equivalent, we will focus on bounding $\Cl_{L/k}[\ell]$ by $\Disc(L/k)$ for the whole paper. 
\end{remark}

\begin{remark}\label{rmk:Delta}
	In most situations in this paper, the parameter $\Delta$ in Definition \ref{bad} will be taken to be $\Delta< \frac{1}{2\ell (d-1)}$ where $d = [L:k]$. We will denote $\Delta(\ell, d)$ for such a number that is very close to $\frac{1}{2\ell(d-1)}$ for simplicity. 
\end{remark}

Then in our language, we will use the following format of this critical lemma throughout the proof of the theorems in section \ref{sec:odd} and \ref{sec:even}:
\begin{lemma}[\cite{EV07}]\label{lem:EVW}
	Given a Galois extension $L/k$, an integer $\ell>1$ with $(\ell, [L:k])=1$, $0< \theta < \frac{1}{2\ell (d-1)}$. If $L/k$ is $\theta$-good with respect to $c$, then
	$$|\Cl_{L/k}[\ell]|\le O_{\epsilon, k,c}(\Disc(L/k)^{1/2-\theta+\epsilon}).$$
\end{lemma}

\section{Odd $p$}\label{sec:odd}
In this section, we work with the elementary abelian groups $A = (\zp)^r$ with $p$ odd and $r>1$. Firstly, in section \ref{ssec:odd-comparable}, section \ref{ssec:odd-incomparable} and section \ref{ssec:odd-rank2-results}, we will focus on the case $r=2$. In section \ref{ssec:odd-induction}, we will apply the result we obtained for $r = 2$ to obtain results for every $r>2$. 

We introduce the notation for section \ref{sec:odd}. For $A = \zp \times \zp$, there are $p+1$ non-trivial subgroups $A_i \simeq \zp$ with $A/A_i \simeq \zp$ for $i = 1, \cdots, p+1$. Therefore given an arbitrary $A$-extension $L/k$, there are $p+1$ non-trivial sub-extensions $K_i/k$. For simplicity of our discussion, we will order $K_i$ by $\Disc(K_i/k)$, i.e., we order them so that
$$\Disc(K_i/k) \le \Disc(K_j/k) \text{ iff } i\le j.$$

We will separate the discussion depending on the size of
\begin{equation}\label{def:eta-odd}
\eta =\eta(L/k):= \frac{\ln \Disc(K_2/k)}{\ln \Disc(K_1/k)}\ge 1.
\end{equation}
We will say $L/k$ is \emph{comparable} if $\eta$ is small, and \emph{incomparable} if $\eta$ is big. We give the proof for the two cases in section \ref{ssec:odd-comparable} and \ref{ssec:odd-incomparable} respectively with two different strategies. The cut-off for the two cases is denoted $\eta_0 = \eta_0(\ell,p)_k$, which is determined in section \ref{ssec:odd-incomparable} (see Theorem \ref{thm:odd-incomparable-overQ}, \ref{thm:odd-incomparable-overk} and Remark \ref{rmk:delta-reasoning-overk}):
\begin{equation}\label{eqn:delta-case}
\begin{aligned}
\eta_0(\ell, p)_k  =  \begin{cases} 
&  ((p-1)\cdot \Delta(\ell,p) \cdot (1-2/p))^{-1} \quad \text{if $k=\Q$;}\\
&   \max\{ \beta(k,\zp), \gamma(k, \zp)+\Delta(\ell,p) \}/\Delta(\ell, p)  \quad \text{if $k\neq \Q$}.
\end{cases}
\end{aligned}
\end{equation}

The power saving $\delta_c$ in section \ref{ssec:odd-comparable} stands for \textbf{c}omparable case and $\delta_{ic}$ in section \ref{ssec:odd-incomparable} stands for \textbf{i}n\textbf{c}omparable case. 

In cases where all parameters $\ell$, $k$ and $p$ are clear, we will write $\eta_0$ instead of $\eta_0(\ell,p)_k$ for simplicity. In cases where $k=\Q$, we will drop $k$ in the notation for simplicity, i.e., we will write $\eta_0(\ell,p)$ instead of $\eta_0(\ell,p)_{\Q}$.

\subsection{Comparable Size}\label{ssec:odd-comparable}
In this section, we will consider $L/k$ with small $\eta$. The approach used in this section will be universally true for any bounded range of $\eta$. For example, we will state the theorem with $\eta\le \eta_0\cdot (1+\epsilon_0) =  \eta_0(\ell,p)_k\cdot ( 1+\epsilon_0)$ where $\epsilon_0>0$ is a small number, and $\eta_0(\ell,p)_k$ is listed in (\ref{eqn:delta-case}). We will use this strategy especially when $\eta$ is small. When $\eta$ is big, we refer to section \ref{ssec:odd-incomparable}. Here the introduction of $\epsilon_0$ is only a technical treatment in order to simplify the dependence on $c$, the constant defined in Definition \ref{bad}.

\begin{theorem}\label{thm:odd-comparable}
	Given $A = \zp\times \zp$, an integer $\ell>1$ with $(\ell, p)=1$ and a number field $k$. For any $A$-extension $L/k$ with $\eta = \eta(L/k) \le \eta_0\cdot (1+\epsilon_0) = \eta_0( \ell, p)_k \cdot (1+\epsilon_0)$, we have the pointwise bound
	$$ |\Cl_{L/k}[\ell]| \le  O_{\epsilon,k,\epsilon_0}(\Disc(L/k)^{1/2 -\delta +\epsilon}).$$
	where 
	$$\delta = \delta_c(\eta, \ell, p)= \frac{\Delta(\ell,p)}{ p(\eta+1)},$$ 
	and $\eta = \eta(L/k) = \frac{\ln \Disc(K_2/k)}{ \ln \Disc(K_1/k)}$. 
\end{theorem}
\begin{proof}
We separate the discussion when $K_1/k$ is $\Delta(\ell, p)$-bad or good with respect to $c$ where $c$ is a fixed absolute number satisfying $c< \frac{1}{p+1}$. The constant $c$ will be fixed once and for all in the proof of the current theorem. By Lemma  \ref{lem:rank2-disc-bound}, we get 
$$\Disc(K_1/k) \ge \Disc(L/k)^{1/p(\eta+1)} \ge \Disc(L/k)^{1/p(\eta_0(1+\epsilon_0)+1)}.$$
So for a fixed integer $L_0>0$, there are only finitely many $A$-extensions $L/k$ where $\Disc(L/k) \le L_0$, thus finitely many $L/k$ with $\Disc(K_1/k)\le K_0 = L_0^{1/p(\eta_0(1+\epsilon_0)+1)}$ and $\eta(L/k) \le \eta_0(1+\epsilon_0)$. So we can assume both $L/k$ and $K_1/k$ are sufficiently large.

If $K_1/k$ is $\Delta(\ell, p)$-bad, then we are going to show that at least one of $K_i/k$ is $\theta_i$-good where $$\theta_i: = \Delta(\ell, p) \frac{\ln \Disc(K_1/k)}{ \ln \Disc(K_i/k)} <\Delta(\ell,p),$$
for $2\le i\le p+1$. Equivalently, we define $\theta_i$ so that $\Disc(K_1)^{\Delta(\ell,p)} = \Disc(K_i)^{\theta_i}$. Consider all primes $p$ in $k$ with $|p|<Y$ where $Y = \Disc(K_1/k)^{\Delta(\ell, p)}$. Since $K_1/k$ is $\Delta(\ell,p)$-bad, there are at most $cY/\ln Y$ primes in $k$ splitting in $K_1/k$. The number of primes in $k$ that are ramified in $L/k$ is bounded by 
$$O_{\epsilon, k}(\Disc(L/k)^{\epsilon})\le O_{\epsilon,k, \epsilon_0}(Y^{\epsilon}),$$
since $Y\ge \Disc(L/k)^{\Delta(\ell,p)/p(\eta_0(1+\epsilon_0)+1)}$. Therefore when $L/k$ is sufficiently large,
\begin{equation}
\begin{aligned}
\pi(Y;K_1/k, \hat{e}) = \pi(Y)- \pi(Y;K_1/k, e) - O_{\epsilon,k,\epsilon_0}(Y^{\epsilon}) \ge (1- c-\epsilon)\cdot \frac{Y}{\ln Y},
\end{aligned}
\end{equation}
where the last inequality holds whenever $Y \ge Y_0 = Y_0(\epsilon, \epsilon_0)$ with $Y_0$ depending at most on $\epsilon$ and $\epsilon_0$. Since the decomposition group of $A$ at an unramified prime is cyclic, a prime $p$ in $k$ that is inert in $K_1/k$ and must be split in some $K_i$ for $2\le i\le p+1$. By pigeon hole principle, there exists at least one $K_i/k$ satisfying
$$ \pi (Y;K_i/k, e) \ge \frac{1-c-\epsilon}{p}\cdot \frac{Y}{\ln Y} \ge c\frac{Y}{\ln Y},$$
then $K_i/k$ is $\theta_i$-good. Let's say $K_j/k$ with $j>1$ is $\theta_j$-good, then by Lemma \ref{lem:EVW}, we get
	$$|\Cl_{K_j/k}[\ell]|\le O_{\epsilon, k}(\Disc(K_j/k)^{1/2-\theta_j+\epsilon}),$$	
where we drop the dependence on $c$ since we fix the absolute number $c<\frac{1}{p+1}$ from the beginning. Therefore by Lemma \ref{lem:class-grp-decomposition} and \ref{lem:disc-prod} and \ref{lem:rank2-disc-bound}, when $\Disc(L/k) \ge L_0(\epsilon,\epsilon_0) = Y_0(\epsilon,\epsilon_0)^{p(\eta_0(1+\epsilon_0)+1)/\Delta(\ell,p)}$, we get
	\begin{equation}
	\begin{aligned}
	|\Cl_{L/k}[\ell]| = & \prod_{i}|\Cl_{K_i/k}[\ell]|\le  O_{\epsilon, k}(\Disc(K_j/k)^{1/2-\theta_i+\epsilon}) \prod_{i\neq j} \Disc(K_i/k)^{1/2+\epsilon}\\
	\le & O_{\epsilon,k}\Big(\frac{ \Disc(L/k)^{1/2+\epsilon}}{ \Disc(K_j/k)^{\theta_i}}\Big) \le O_{\epsilon,k}(\Disc(L/k)^{1/2- \Delta(\ell,p)/p(\eta+1) +\epsilon}).\\
	\end{aligned}
	\end{equation}
	
If $K_1$ is $\Delta(\ell, p)$-good, then we get from Lemma \ref{lem:EVW} that
	$$|\Cl_{K_1/k}[\ell]| \le O_{\epsilon,k}( \Disc(K_1/k)^{1/2-\Delta(\ell,p) +\epsilon}). $$
Then similarly, by Lemma \ref{lem:class-grp-decomposition} and \ref{lem:disc-prod} and \ref{lem:rank2-disc-bound}, we get
		\begin{equation}
		\begin{aligned}
		|\Cl_{L/k}[\ell]| = &\prod_{i}|\Cl_{K_i/k}[\ell]|\le  O_{\epsilon, k}\Big(\Disc(K_1/k)^{1/2-\Delta(\ell,p)+\epsilon}\Big) \prod_{i\neq 1} \Disc(K_i/k)^{1/2+\epsilon}\\
		\le & O_{\epsilon,k}\Big(\frac{ \Disc(L/k)^{1/2+\epsilon}}{ \Disc(K_1/k)^{\Delta(\ell,p)}}\Big) \le O_{ \epsilon,k}(\Disc(L/k)^{1/2-\Delta(\ell,p)/p(\eta+1)+\epsilon}).\\
		\end{aligned}
		\end{equation}
Since we assume $L/k$ sufficiently large for later discussion, i.e., $\Disc(L/k)\ge L_0(\epsilon,\epsilon_0)$, in summary, we show that for any $A$-extension $L/k$
	$$ |\Cl_{L/k}[\ell]| \le  O_{\epsilon,k,\epsilon_0}(\Disc(L/k)^{1/2 -\delta +\epsilon}),$$
with $\delta = \delta_c(\eta, \ell, p) =  \frac{\Delta(\ell,p)}{p(\eta+1)}$. 
\end{proof}

This gives a power saving on the pointwise bound of $\Cl_{L/k}[\ell]$ in terms of $\eta(L/k)$. 
\begin{remark}
	Notice that here in Theorem \ref{thm:odd-comparable} we only take the bound $\eta_0(\ell,p)_k\cdot (1+\epsilon_0)$ for $\eta$ for simplicity. The same non-trivial saving $\delta = \delta_c(\eta, \ell, p)$ can be obtained with $\eta\le M$ for arbitrary number $M$. In this scenario, the implied constant depends on $M$ instead of $\epsilon_0$.
\end{remark}

\subsection{Incomparable Size}\label{ssec:odd-incomparable}
In this section, we will give another strategy when $\eta$ is very large, equivalently when $K_2$ is much large than $K_1$. We will also see the cut-off $ \eta_0(\ell, p)_k$ from the following theorem. We will first prove the result over $\mathbb{Q}$ in Theorem \ref{thm:odd-incomparable-overQ} and then prove the result over a general number field $k$ in Theorem \ref{thm:odd-incomparable-overk}.

\begin{theorem}\label{thm:odd-incomparable-overQ}
	Given $A = \zp\times \zp$ with odd $p$, an integer $\ell>1$ with $(\ell, p)=1$. Denote $\eta_0 = \eta_0(\ell,p) = ((p-1)\cdot \Delta(\ell,p) \cdot (1-2/p))^{-1}$. For any $A$-extension $L/\Q$ with $\eta = \eta(L/\Q) > \eta_0(1+\epsilon_0)$, we have the pointwise bound
	$$ |\Cl_L[\ell]| \le O_{\epsilon,\epsilon_0}(\Disc(L)^{1/2 - \delta +\epsilon})$$
	for some $$\delta = \delta_{ic}(\eta,\ell,p) = \frac{\Delta(\ell,p) \eta}{ p(\eta+1)}$$ where $\eta = \frac{\ln \Disc(K_2)}{\ln \Disc(K_1)}$.
\end{theorem}
\begin{proof}	
	By Lemma \ref{lem:rank2-disc-bound}, we have $ \Disc(K_2) \ge \Disc(L)^{ \eta/p(\eta+1)} \ge \Disc(L)^{\eta_0(1+\epsilon_0)/p(\eta_0(1+\epsilon_0)+1)}$. So for a fixed integer $L_0>0$, there are only finitely many $L$ with $\Disc(L)\le L_0$, thus finitely many $L$ with $\Disc(K_2)\le K_0 = L_0^{\eta_0(1+\epsilon_0)/p(\eta_0(1+\epsilon_0)+1)}$ and $\eta>\eta_0(1+\epsilon_0)$. So we can assume that both $L$ and $K_2$ are sufficiently large. 
	
    We will show that at least one of $K_i$ for $2\le i\le p+1$ is  $\theta_i$-good for some $\theta_i>0$ with respect to $c$ where $c$ is a fixed small number satisfying $c< \frac{(p-2)\epsilon_0}{2+p\epsilon_0}$. The constant $c = c(\epsilon_0)$ will be fixed once and for all for the current theorem.
    
    If $\eta(L/\Q) > \eta_0(1+\epsilon_0)$, then we can apply Lemma \ref{lem:B-T} with 
		$$x = \Disc(K_2)^{\Delta(\ell,p)}, \quad \quad q = \Cond(K_1)\asymp  \Disc(K_1)^{1/(p-1)},$$
	to count the number of primes in $\Q$ splitting in $K_1/\Q$. By class field theory, this is equivalent to taking $\frac{\phi(q)}{p}$ residue classes $a \mod q$ and then adding up $\pi(x;q, a)$ over $a$. Therefore we have positive density of primes up to $x$ in $\Q$ that are inert in $K_1/\Q$, 
	\begin{equation}
	\begin{aligned}
	\pi(x; K_1/\Q, \hat{e}) & = \pi(x) - \pi(x; K_1/\Q, e) - O_{\epsilon}(\Disc(L)^{\epsilon}) \\
	& \ge \pi(x) - \frac{2}{1 - 1/\Delta(\ell,p)(p-1)\eta}\cdot \frac{x}{p\ln x} - O_{\epsilon,\epsilon_0}(x^{\epsilon}),\\
	& \ge C \frac{x}{\ln x}.
	\end{aligned}
	\end{equation}
	The first inequality comes from Lemma \ref{lem:B-T} and $\Disc(K_2) \ge \Disc(L)^{\eta_0(1+\epsilon_0)/p(\eta_0(1+\epsilon_0)+1)}$. The second inequality holds when we take $C = 1 - \frac{2}{p} \frac{1}{1 - 1/\Delta(\ell,p)(p-1)\eta_0(1+\epsilon_0)} -\epsilon$ and $x \ge x_0 = x_0(\epsilon)$ with $x_0$ depending at most on $\epsilon$. Primes that are inert in $K_1$ must be split in $K_i$ for some $i>1$. Therefore by pigeon hole principle, there exists at least one $K_j$ for $2\le j\le p+1$ satisfying
    \begin{equation}
    \pi(x; K_j, e) \ge \frac{C}{p} \cdot \frac{x}{\ln x}\ge c\frac{x}{\ln x},
    \end{equation}
    where the last inequality comes from the assumption $c < \frac{(p-2)\epsilon_0}{2+\epsilon_0}$. This $K_j$ is $\theta_j$-good for
	\begin{equation}
		\theta_j:= \Delta(\ell,p) \cdot \frac{\ln \Disc(K_2)}{\ln \Disc(K_j)} \le \Delta(\ell, p). 
	\end{equation}
	
	Then by Lemma \ref{lem:EVW}, we get
	$$|\Cl_{K_j}[\ell]| \le O_{\epsilon,c}(\Disc(K_j)^{1/2-\theta_j+\epsilon}) = O_{\epsilon,\epsilon_0}(\Disc(K_j)^{1/2-\theta_j+\epsilon}),$$
	since our constant $c$ is a small number depending at most on $\epsilon_0$. By Lemma \ref{lem:disc-prod} and \ref{lem:class-grp-decomposition} and \ref{lem:rank2-disc-bound}, we have for every $L$ that
		\begin{equation}
		\begin{aligned}
		|\Cl_L[\ell]|  \le O_{\epsilon,\epsilon_0}(\Disc(L)^{1/2- \Delta(\ell,p)\eta/p(\eta+1)+\epsilon }).\\
		\end{aligned}
		\end{equation}
	So we prove this theorem with 
	$$\delta_{ic}(\eta, \ell, p) = \frac{\Delta(\ell,p) \eta}{p(\eta+1)}.$$
\end{proof}
Then we give the version over a general number field. The only distinction is that we will apply Lemma \ref{lem:ZT} instead of Lemma \ref{lem:B-T}.
\begin{theorem}\label{thm:odd-incomparable-overk}
	Given $A = \zp\times \zp$, an integer $\ell>1$ with $(\ell, p)=1$. Denote $\eta_0 = \eta_0(\ell,p)_k = \max\{ \beta, \gamma +\Delta(\ell,p) \}/\Delta(\ell, p)$ where $\beta = \beta(k, \zp)$ and $\gamma= \gamma(k, \zp)$. For any $A$-extension $L/k$ with $\eta(L/k) > \eta_0$, we have the pointwise bound
    $$	|\Cl_{L/k}[\ell]|  \le O_{\epsilon,k}(\Disc(L/k)^{1/2-\delta+\epsilon }),$$
    where
   $$\delta= \delta_{ic,k}(\eta, \ell,p) = \frac{(\Delta(\ell,p)-\gamma/\eta)\eta}{p(\eta+1)}.$$
\end{theorem}
\begin{proof}
	Notice that by Lemma \ref{lem:rank2-disc-bound}, we have 	
	$$\Disc(K_2/k)\ge \Disc(L/k)^{\eta/p(\eta+1)} \ge \Disc(L/k)^{\eta_0/p(\eta_0+1)}.$$
	So for a fixed integer $L_0>0$, there are only finitely many $L/k$ with $\Disc(L/k)\le L_0$, thus finitely many $L/k$ with $\Disc(K_2/k)\le K_0= L_0^{\eta_0/p(\eta_0+1)}$ and $\eta> \eta_0$. So we can assume that both $K_2/k$ and $L/k$ are sufficient large.
	
	Firstly, we will show that there exist a lot of primes inert in $K_1/k$ with the range of consideration $x =\Disc(K_2/k)^{\Delta(\ell, p)}$ when $L/k$ is sufficiently large. We will apply Lemma \ref{lem:ZT} to $K_1/k$ with $x = \Disc(K_2/k)^{\Delta(\ell, p)}$. Recall the absolute constant $D_0=D_0(k)$ depending at most on $k$ in Lemma \ref{lem:ZT}. 
	
	If $\Disc(K_1/k)<D_0$, then it follows from the standard Chebotarev density theorem that for $C'= \frac{p-1}{p}-\epsilon$, we have
	$$\pi(x; K_1/k, \hat{e}) \ge C'\frac{x}{\ln x} = \frac{C'}{\Delta(\ell,p)}\cdot \frac{\Disc(K_2/k)^{\Delta(\ell, p) }}{\ln \Disc(K_2/k)},$$
    when $x$ is sufficiently large comparing to $D_0$, say $x \ge x_0 = x_0(D_0,\epsilon) = x_0(k,\epsilon)$ where $x_0$ depends at most on $D_0$ and $\epsilon$, thus depends at most on $k$ and $\epsilon$. If we take $K_0^{\Delta(\ell,p)} = x_0(k,\epsilon)$, then when $\Disc(L/k) \ge L_0 (k,\epsilon)= K_0(k,\epsilon)^{p(\eta_0+1)/\eta_0}$ is sufficiently large, we know that if $\Disc(K_1/k)< D_0$ then $\pi(x; K_1/k, \hat{e}) \ge \frac{C'}{\Delta(\ell,p)}\cdot \frac{\Disc(K_2/k)^{\Delta(\ell, p) }}{\ln \Disc(K_2/k)}$.

     If $\Disc(K_1/k)\ge D_0(k)$, then we apply Lemma \ref{lem:ZT}. When $\eta>\eta_0$, we have $\Disc(K_2/k)^{\Delta(\ell, p)} \ge \max\{ \Disc(K_1/k)^{\beta}, \Disc(K_1/k)^{\gamma} \}$ for $\beta = \beta(k, \zp)$ and $\gamma= \gamma(k, \zp)$ in Lemma \ref{lem:ZT}. By Lemma \ref{lem:ZT} there exists some $C_k>0$ such that
		\begin{equation}
		\pi(x; K_1/k, \hat{e}) \ge  C_k\frac{1}{\Disc(K_1/k)^{\gamma}} \cdot \frac{x}{\ln x} \ge \frac{C_k}{\Delta(\ell,p)}\cdot \frac{\Disc(K_2/k)^{\Delta(\ell, p) - \gamma/\eta}}{\ln \Disc(K_2/k)},
		\end{equation}
	where $C_k$ is some constant only depending on $k$. So in summary, as $L/k$ is sufficiently large (i.e., $\Disc(L/k)\ge L_0(k,\epsilon) = K_0(k,\epsilon)^{p(\eta_0+1)/\eta_0}$), we show
	$$\pi(x; K_1/k, \hat{e}) \ge \frac{C''_k}{\Delta(\ell,p)}\cdot \frac{\Disc(K_2/k)^{\Delta(\ell, p) - \gamma/\eta}}{\ln \Disc(K_2/k)},$$
	where $C''_k = \min\{ C', C'_k\}$ depends only on $k$.
	
	By pigeon hole principle, there exists at least one $K_j/k$ for $2\le j\le p+1$ where 
	   \begin{equation}
	   \pi(x; K_j/k, e) \ge \frac{C''_k}{p\Delta(\ell,p)} \cdot \frac{\Disc(K_2/k)^{\Delta(\ell, p) - \gamma/\eta}}{\ln \Disc(K_2/k)}.
	   \end{equation}
	 Finally by Lemma \ref{lem:EVW} and Lemma \ref{lem:rank2-disc-bound}, we have for any $L/k$ that
	 	\begin{equation}
	 	\begin{aligned}
	 	|\Cl_{L/k}[\ell]|  \le O_{\epsilon,k}(\Disc(L/k)^{1/2-\delta+\epsilon }),\\
	 	\end{aligned}
	 	\end{equation}
	 where
	 $$ \delta= \delta_{ic,k}(\eta, \ell, p) = \frac{(\Delta(\ell,p)-\gamma/\eta)\eta}{p(\eta+1)}.$$
\end{proof}
\begin{remark}
	Here when $k=\Q$ and $p=2$, we can apply Lemma \ref{lem:Maynard} as a sub-case of Lemma \ref{lem:ZT} with $\gamma(\Q, \mathbb{Z}/2\mathbb{Z}) = 1/2-\epsilon$, $\beta(\Q, \mathbb{Z}/2\mathbb{Z} ) = 8$ and $D_0=q_2$. 
\end{remark}

\subsection{Savings for Odd $A$ with Rank $2$ }\label{ssec:odd-rank2-results}
So combining Theorem \ref{thm:odd-comparable} and \ref{thm:odd-incomparable-overQ} and (\ref{eqn:abs-to-rel}) in Remark \ref{rmk:abs-to-rel}, we get the following theorem. 
\begin{theorem}[Odd Exponent, Rank $2$, Over $\Q$]\label{thm:odd-rank2-final-Q}
	Given $A = \zp\times \zp$ with odd $p$ and an integer $\ell>1$ with $(\ell, p)=1$. For any $A$-extension $L/\Q$, we have
	$$ |\Cl_L[\ell]| \le O_{\epsilon}(\Disc(L)^{1/2 - \delta( \ell,p)+\epsilon}),$$
	with $$\delta(\ell,p)= \delta_c(\eta_0, \ell, p)=\frac{\Delta(\ell, p)}{p(1+\eta_0)},$$
	where $\eta_0 = \frac{1}{(p-1) \Delta(\ell, p)(1-2/p)}$. 
\end{theorem}
\begin{proof}
	Combining Theorem \ref{thm:odd-comparable} and Theorem \ref{thm:odd-incomparable-overQ}, for every fixed small $\epsilon_0$, we show that for every $A$-extension $L/\Q$
	$$ |\Cl_L[\ell]| \le O_{\epsilon}(\Disc(L)^{1/2 - \delta(\ell,p, \epsilon_0)+\epsilon}),$$
	where $\delta(\ell,p,\epsilon_0) = \delta_c(\eta_0(1+\epsilon_0), \ell, p) = \frac{\Delta(\ell,p)}{p(1+\eta_0(1+\epsilon_0))}$. Since we can take arbitrarily small $\epsilon_0$ and we also state the theorem with arbitrarily small $\epsilon$, we can get
	$$ |\Cl_L[\ell]| \le O_{\epsilon}(\Disc(L)^{1/2 - \delta(\ell,p)+\epsilon}),$$
	for $\delta(\ell,p) = \delta_c(\eta_0, \ell, p)$. 
\end{proof}

For general number field $k$, similarly notice that since $\delta_{ic, k}(\eta, \ell,p)$ always increases as $\eta$ increases and $\delta_{c}(\eta, \ell,p)$ always decreases as $\eta$ increases. By comparing $\delta_c(\eta_0, \ell, p)$ and $\delta_{ic,k}(\eta_0, \ell, p)$ at $\eta_0 = \eta_0(\ell,p)_k =  \max\{ \beta(k, \zp), \gamma(k, \zp) +\Delta(\ell,p) \}/\Delta(\ell, p)$, we see that the smallest saving always happens at $\delta_c(\eta_0, \ell, p)_k$. So
 we are guaranteed to find the universal saving $\delta>0$ for all ranges of $\eta$ at the cut-off $\eta_0$. 
\begin{remark}\label{rmk:delta-reasoning-overk}
In the proof of Theorem \ref{thm:odd-incomparable-overk}, we can see that it suffices to take $\eta_0(\ell,p)_k$ to be  $ \max\{ \beta(k, \zp), \gamma(k, \zp) \}/\Delta(\ell, p)$. The reason that instead we take $$\eta_0(\ell,p)_k =  \max\{ \beta(k, \zp), \gamma(k, \zp) +\Delta(\ell,p) \}/\Delta(\ell, p),$$ is that it guarantees $\delta_{ic,k}(\eta_0, \ell, p)> \delta_c(\eta_0, \ell, p)$ and simplifies the final expression of the saving. However, notice that usually $\beta$ is larger than $\gamma$ in reality, see \cite{ZamThesis} for example, so in such situations it will not change the actual value of $\eta_0(\ell, p)_k$ after plugging in $\beta$ and $\gamma$. 
\end{remark}

\begin{theorem}[Odd Exponent, Rank $2$, Over $k$]\label{thm:odd-rank2-final-k}
	Given $A = \zp\times \zp$ with odd $p$ and an integer $\ell>1$ with $(\ell, p)=1$. For any $A$-extension $L/k$, we have
	$$ |\Cl_L[\ell]| \le O_{\epsilon,k}( \Disc(L)^{1/2 - \delta_k(\ell,p)+\epsilon}),$$
	where $\delta_k(\ell, p) =\delta_c(\eta_0, \ell, p) = \frac{\Delta(\ell, p)}{p(1+\eta_0)}$ and $\eta_0 = \eta_0(\ell,p)_k = \max\{ \beta(k, \zp), \gamma(k, \zp)+\Delta(\ell,p) \}/\Delta(\ell, p)$. 
\end{theorem}

\subsection{Induction}\label{ssec:odd-induction}
In this section, we will derive the $\ell$-torsion bound for every $A = (\zp)^r$ when $r>2$ from the case $A = \zp\times \zp$.

\begin{theorem}[Odd Exponent, Over $k$]\label{thm:odd-final-results-overk}
	Given $A= (\zp)^r$ with $r\ge 2$ and $p$ odd. Given an arbitrary integer $\ell = \ell_{ (p)}\cdot \ell_{p}$ where $\ell_{(p)}$ is the maximal factor of $\ell$ relatively prime to $p$. For any $A$-extension $L/k$, we have
	$$ |\Cl_{L}[\ell]| \le O_{k, \epsilon}(\Disc(L)^{1/2 - \delta_k(\ell_{(p)},p)+\epsilon}),$$
	where $\delta_{\Q}(\ell,p)= \delta(\ell,p)$ in Theorem \ref{thm:odd-rank2-final-Q} when $k=\Q$, and $\delta_k(\ell,p)$ in Theorem \ref{thm:odd-rank2-final-k} for general $k$. 
\end{theorem}
\begin{proof}
	Firstly we assume $(\ell, p)=1$. The result for $r=2$ and $(\ell, p)$ is stated in Theorem \ref{thm:odd-rank2-final-Q} and \ref{thm:odd-rank2-final-k}. For $r>2$ and $(\ell, p)=1$, notice that
	\begin{equation}
	\begin{aligned}
	|\Cl_{L/k}[\ell]| &= \prod_{i} |\Cl_{K_i/k}[\ell]| =  \Big(\prod_{j} |\Cl_{M_j/k}[\ell]|\Big)^{1/(p+1)}  \le O_{\epsilon,k}\Big(\prod_j \Disc(M_j/k)^{1/2-\delta(\ell,p)+\epsilon}\Big)^{1/(p+1)} \\
	& =O_{\epsilon,k}\Big(\prod_j \Disc(M_j/k)^{1/(p+1)}\Big)^{1/2-\delta(\ell, p) +\epsilon} = O_{\epsilon,k}(\Disc(L/k)^{1/2-\delta(\ell,p)+\epsilon}). 
	\end{aligned}
	\end{equation}
	where $M_j$ ranges over all degree $p^2$ sub-extensions in $L$ over $\Q$. The first equality comes from Lemma \ref{lem:class-grp-decomposition}. The second equality comes from Corollary \ref{coro:class-grp-second-layer-decomposition}. The first inequality comes from Theorem \ref{thm:odd-rank2-final-Q}. The last equality comes from (\ref{eqn:disc-prod-second-layer}). Finally it follows from (\ref{eqn:abs-to-rel}) in Remark \ref{rmk:abs-to-rel}. 
	
	For general $\ell = \ell_{(p)}\ell_{p}$, notice that $|\Cl_L[\ell]|= |\Cl_L[\ell_{( p)}]|\cdot |\Cl_L[\ell_p]|$ and $|\Cl_L[\ell_p]| \le O_{\epsilon}(\Disc(L)^{\epsilon})$, we get $|\Cl_{L}[\ell]| \le O_{k, \epsilon}(\Disc(L)^{1/2 - \delta_k(\ell_{(p)},p)+\epsilon}).$	
\end{proof}

\begin{remark}[Odd Exponent, $\ell=2$, Over $k$]\label{rmk:odd-final-k-2torsion}
	When $\ell=2$, we can obtain better results because of the pointwise result on $2$-torsion from \cite{BSTTTZ}. It is proved that $|\Cl_F[2]| \le O(\Disc(F)^{1/2-1/2d+\epsilon})$ where $d=[F:\Q]$ by \cite{BSTTTZ}. By (\ref{eqn:abs-to-rel}) in Remark\ref{rmk:abs-to-rel}, we get for $K$ with $\Gal(K/k)=\zp$, the $2$-torsion is bounded 
	$$|\Cl_{K/k}[2]| \le O_{\epsilon, k}( \Disc(L/k)^{1/2-1/2p+\epsilon}).$$ 
	Then the statement follows from a straight forward use of Lemma \ref{lem:class-grp-decomposition}. 
\end{remark}

\section{Even $p$}\label{sec:even}
In this section, we will discuss the cases when $A$ is an elementary abelian group with even exponent, i.e., when $A = (\mathbb{Z}/2\mathbb{Z})^r$ and $r>1$. In section \ref{ssec:even-rank2}, we first give the result for $r=2$. Then in order to get a better saving than that obtained in section \ref{ssec:even-rank2}, we focus on $r=3$ in section \ref{ssec:even-rank3-comparable}, \ref{ssec:even-rank3-incomparable} and \ref{ssec:even-rank3-results}, and use an induction to get an overall better saving for $r>3$ in section \ref{ssec:even-induction}. 

The main reason that we separate the discussion for $p$ being odd and even is that in Theorem \ref{thm:odd-incomparable-overQ} we ask the constant $c$ to be smaller than $\frac{(p-2)\epsilon_0}{2+\epsilon_0}$, which is only positive when $p$ is odd. So when $p = 2$, we need to replace Theorem \ref{thm:odd-incomparable-overQ}, and, more importantly, consequences of Theorem \ref{thm:odd-incomparable-overQ}. The strategy for doing this is treat $r=3$ as the initial case for $p=2$, i.e., we replace Theorem \ref{thm:odd-incomparable-overQ} with Theorem \ref{thm:even-rank3-incomparable} in this section.

\subsection{Even Exponent with Rank $2$}\label{ssec:even-rank2}
In this section, we work with $A = \mathbb{Z}/2\mathbb{Z} \times  \mathbb{Z}/2\mathbb{Z}$. We will follow the notation introduced at the beginning of section \ref{sec:odd}. Recall that we have $K_i$ for $i = 1,2,3$ where $\Disc(K_1/k) \le \Disc(K_2/k) \le \Disc(K_3/k)$, and $\eta(L/k):= \frac{\ln \Disc(K_2/k)}{\ln \Disc(K_1/k)}$. Again we split the discussion to $\eta$ being small (the comparable case) and $\eta$ being big (the incomparable case). We take $$\eta_0 = \eta_0(\ell,2)_k = \max\{ \beta(k, \zp), \gamma(k, \zp)+\Delta(\ell,2) \}/ \Delta(\ell,2)$$ in this section.  

For the comparable case, we recall Theorem \ref{thm:odd-comparable} (which is stated for all $A$, not just odd $A$), which states that
$$|\Cl_{L/k}[\ell]| \le O_{\epsilon,k,\epsilon_0}(\Disc(L/k)^{1/2-\delta+\epsilon}),$$
where $\delta = \delta_c(\eta, \ell, 2) = \frac{\Delta(\ell, 2)}{2(\eta+1)}$ and $\eta = \eta(L/k)=\frac{\ln \Disc(K_2/k)}{\ln \Disc(K_1/k)} \le \eta_0(\ell, 2)_k(1+\epsilon_0)$. 

For the incomparable case, we recall Theorem \ref{thm:odd-incomparable-overk} (which is stated for all $A$, not just odd $A$), which states that
$$|\Cl_{L/k}[\ell]| \le O_{\epsilon, k}( \Disc(L/k)^{1/2-\delta+\epsilon}),$$
where $\delta = \delta_{ic, k}(\eta, \ell, 2) = \frac{(\Delta(\ell,2)-1/\eta) \eta}{2(\eta+1)}$ when $\eta> \eta_0(\ell,2)_k$. Combining the two cases, we get the following theorem. 

\begin{theorem}[Even Exponent, Rank $2$, Over $k$]\label{thm:even-rank2-results-overk}
	Given $A = \mathbb{Z}/2\mathbb{Z}\times \mathbb{Z}/2\mathbb{Z}$ and $\ell>1$ an odd integer. For any $A$-extension $L/k$, we have
	$$|\Cl_L[\ell]| \le O_{\epsilon, k}(\Disc(L)^{1/2-\delta_k(\ell,2)+\epsilon}),$$
	with 
	$$\delta_k(\ell, 2) =\frac{\Delta(\ell, 2)}{p(\eta_0+1)},$$
	where $\eta_0 = \max\{ \beta(k, \mathbb{Z}/2\mathbb{Z}), \gamma(k, \mathbb{Z}/2\mathbb{Z}) +\Delta(\ell,2) \}/\Delta(\ell, 2)$. 
	In particular, when $k = \Q$, we have
	$$\delta_{\Q}(\ell, 2) = \frac{\Delta(\ell, 2)}{p(\eta_0+1)} = \frac{1}{64\ell^2+4\ell}.$$
\end{theorem}
\begin{proof}
	If $k = \Q$, by Lemma \ref{lem:Maynard}, we can take $\beta(\Q, \mathbb{Z}/2\mathbb{Z}) = 8$ and $\gamma(\Q, \mathbb{Z}/2\mathbb{Z}) = 1/2-\epsilon$. Then $\eta_0(\ell, 2)_{\Q} = \frac{8}{\Delta(\ell, 2)}$. By comparing $\frac{\Delta(\ell, 2)}{p(\eta_0+1)}$ and $\frac{(\Delta(\ell,2)-\gamma(\Q, \mathbb{Z}/2\mathbb{Z}) /\eta_0) \eta_0}{p(\eta_0+1)}$, we see that a universal saving is 
	$$\delta_{\Q}(\ell,2) = \frac{\Delta(\ell, 2)}{p(\eta_0+1)} = \frac{1}{64\ell^2+4\ell}.$$
	
	Similarly, we have
	$$\delta_{k}(\ell,2) = \frac{\Delta(\ell, 2)}{p(\eta_0+1)},$$
	where $\eta_0 = \max\{ \beta(k, \mathbb{Z}/2\mathbb{Z}), \gamma(k, \mathbb{Z}/2\mathbb{Z})+\Delta(\ell,2) \}/\Delta(\ell, 2)$. 
\end{proof}

\subsection{Comparable Size for Rank $3$}\label{ssec:even-rank3-comparable}
In this section, we work with $A =  (\mathbb{Z}/2\mathbb{Z})^r$ with $r>2$. In section \ref{ssec:even-rank3-comparable}, \ref{ssec:even-rank3-incomparable} and \ref{ssec:even-rank3-results}, we focus on the case $r=3$ over $\Q$. In section \ref{ssec:even-induction}, we apply the result we obtained for $r=3$ to obtain results for $r>3$. The main reason that we can get a better saving here for $r=3$ over $\Q$ than $r=2$ is that we can apply the Lemma \ref{lem:B-T} for the incomparable case of $A =  (\mathbb{Z}/2\mathbb{Z})^3$ instead of Lemma \ref{lem:Maynard}. 

We introduce the notation for the current section and section \ref{ssec:even-rank3-incomparable}. For $A = (\mathbb{Z}/2\mathbb{Z})^3$, there are $7$ index-$2$ subgroups and $7$ index-$4$ subgroups. For an $A$-extension $L/\Q$, we denote $M_1$ to be the quartic subfield with smallest discriminant, and $K_m$ to be the smallest quadratic field outside $M_1$. 
Denote $K_i$ for $i= 1,2,3$ to be subfields of $M_1$ ordered by $\Disc(K_i)$. Denote $K'_i$ to be the other quadratic subfield of the compositum $K_mK_i$. So we always have $\Disc(K'_i) \ge \Disc(K_m)$. 
In this section and section \ref{ssec:even-rank3-incomparable} and \ref{ssec:even-rank3-results}, we will denote 
\begin{equation}\label{def:eta-rank3}
\eta= \eta(L/k):= \frac{\ln \Disc(K_m)}{\ln \Disc(M_1)},\quad\quad \eta_0 = \frac{1}{\Delta(\ell,2)}.
\end{equation}
See Theorem \ref{thm:even-rank3-incomparable} for the reason on the choice of $\eta_0. $ We will use $\delta'_c(\eta,\ell)$ and $\delta'_{ic}(\eta, \ell)$ to denote the savings in section \ref{ssec:even-rank3-comparable} and \ref{ssec:even-rank3-incomparable} to distinguish from $\delta_c(\eta, \ell, 2)$ and $\delta_{ic}(\eta,\ell,2)$ used in section \ref{ssec:even-rank2}. 

\begin{theorem}\label{thm:even-rank3-comparable}
		Given $A = (\mathbb{Z}/2\mathbb{Z})^3$ and an odd integer $\ell>1$. For any $A$-extension $L/\Q$ with $\eta(L/\Q) \le \eta_0(1+\epsilon_0) = \frac{1+\epsilon_0}{\Delta(\ell,2)}$, we have 
		$$ |\Cl_{L}[\ell]| \le O_{\epsilon, \epsilon_0}(\Disc(L)^{1/2 - \delta+\epsilon}),$$
		for $$\delta = \delta'_c(\eta,\ell) =  \frac{\Delta(\ell, 4)}{4\eta+2}>0,$$ where $\eta = \frac{\ln \Disc(K_m)}{\ln \Disc(M_1)}$. 
\end{theorem}
\begin{proof}
	The proof is similar with that of Theorem \ref{thm:odd-comparable}. We separate the discussion for $M_1$ being $\Delta(\ell, 4)$-bad or not with respect to $c$ where $c$ is a fixed small number satisfying $c<1/7$. We fix $c$ once and for all for the current theorem. By Lemma \ref{lem:rank3-disc-bound}, we have $\Disc(M_1)\ge \Disc(L)^{1/(4\eta+2)} \ge \Disc(L)^{1/(4\eta_0(1+\epsilon_0)+2)}$.
	So for a fixed $L_0>0$, there are finitely many $L/Q$ with $\Disc(L) \ge L_0$, and thus finitely many $\Disc(M_1)\ge M_0 = L_0^{1/(4\eta_0(1+\epsilon_0)+2)}$ with $\eta(L/\Q)\le \eta_0(1+\epsilon_0)$. So we can assume both $M_1$ and $L$ are sufficiently large.
	
	If $M_1$ is $\Delta(\ell,4)$-good, then by Lemme \ref{lem:class-grp-decomposition} and Lemma \ref{lem:rank3-disc-bound}, we have
	\begin{equation}
	\begin{aligned}
	|\Cl_L[\ell]| &=| \Cl_{M_1}[\ell]| \prod_{K_i\not\subset M_1 }|\Cl_{K_i}[\ell]| \le O_{\epsilon}(\Disc(L)^{1/2-\Delta(\ell,4)/(4\eta+2)+\epsilon}).\\
	\end{aligned}
	\end{equation}

	    If $M_1$ is $\Delta(\ell,4)$-bad, then we have for $x = \Disc(M_1)^{\Delta(\ell,4)}$ that 
			$$\pi (x; M_1, \hat{e}) \ge (1-c-\epsilon)\cdot \frac{x}{\ln x},$$
		when $x \ge x_0(\epsilon, \epsilon_0)$ is sufficiently large with $x_0$ depending at most on $\epsilon$ and $\epsilon_0$. These primes are inert in $M_1/k$, so will always split at exactly $2$ of $\{ K_m, K'_1, K'_2, K'_3\}$ not contained in $M_1$. Denote 
	    $$\theta_i = \frac{\Delta(\ell, 4)\ln \Disc(M_1)}{\ln \Disc(K'_i)},\quad i= 1,2,3,\quad \theta_m = \frac{\Delta(\ell, 4)\ln \Disc(M_1)}{\ln \Disc(K_m)},$$	    
	    for $K'_i$ ($i=1,2,3$) and $K_m$ respectively. By pigeon hole principle, we get at least $\frac{1-c}{{4\choose 2}} \frac{x}{\ln x}$ many primes that are all split in two of $S$. Since $c<1/7$, we get at least two of $K_i$ of $S$ that are $\theta_i$-good. Denote them by $K_j$ for $j\in J$. Therefore when $\Disc(L/k) \ge L_0(\epsilon,\epsilon_0) = x_0(\epsilon, \epsilon_0)^{ (4\eta_0(1+\epsilon_0) +2)/\Delta(\ell,2)}$, we always get for two $K_j$ that
	    	$$|\Cl_{K_j}[\ell]|\le O_{\epsilon}(\Disc(K_j)^{1/2-\theta_j+\epsilon}),$$
	    and it follows that for every $L$ we get
	   \begin{equation}
	   |\Cl_{L}[\ell]| = \prod_{i\notin J} |\Cl_{K_i}[\ell]| \prod_{j\in J} |\Cl_{K_j}[\ell]|\le O_{\epsilon,\epsilon_0}( \Disc(L)^{1/2 - 2\Delta(\ell,4)/(4\eta+2) + \epsilon}),
	    \end{equation}
	    where the last inequality follows from Lemma \ref{lem:rank3-disc-bound}. 	
		Therefore we can always get a saving with
		$$\delta'_c(\eta, \ell) = \frac{\Delta(\ell, 4)}{4\eta +2}.$$
\end{proof}

\subsection{Incomparable Size for Rank $3$}\label{ssec:even-rank3-incomparable}
In this section, we will treat the case when $A =( \mathbb{Z}/2\mathbb{Z})^3$ and the base field is $\Q$, and $\eta(L/\Q)$ is large. 
\begin{theorem}\label{thm:even-rank3-incomparable}
		Given $A =( \mathbb{Z}/2\mathbb{Z})^3$ and an odd integer $\ell>1$. For any $A$-extension $L/k$, if $\eta > \eta_0 (1+\epsilon_0) =  \frac{1+\epsilon_0}{\Delta(\ell, 2)}$, then 
		$$ |\Cl_L[\ell]| \le O_{\epsilon, \epsilon_0}(\Disc(L)^{1/2 - \delta +\epsilon}),$$
		for 
		$$\delta = \delta'_{ic}(\eta, \ell) = \frac{\Delta(\ell,2)\eta}{2\eta+1} >0.$$
\end{theorem}
\begin{proof}
	Similarly with the proof of Theorem \ref{thm:odd-incomparable-overQ}, by Lemma \ref{lem:rank3-disc-bound}, we can assume both $L$ and $K_m$ are sufficiently large.
	
	We will show that at least $2$ of quadratic fields $K_i$ in $\{ K_m, K'_1, K'_2, K'_3 \}$ are $\theta_i$ good for 
	  $$\theta_i = \frac{\Delta(\ell, 2)\ln \Disc(K_m)}{\ln \Disc(K'_i)}, \quad i= 1,2,3,\quad \theta_m = \Delta(\ell, 2),$$
	 with respect to $c$ where $c$ is a small number satisfying $c< \frac{\epsilon_0}{6(1+2\epsilon_0)}$. We will fix $c = c(\epsilon_0)$ once and for all for the current theorem.
	
	We apply Lemma \ref{lem:B-T} with 
		$$x = \Disc(K_m)^{\Delta(\ell,2)}, \quad \quad q = \Cond(M_1)\asymp  \Disc(M_1)^{1/2},$$
	to count the number of primes in $\Q$ that split in $M_1/\Q$. By class field theory, this is equivalent to take $\frac{\phi(q)}{p}$ residue classes $a (\text{mod } q)$ and then add up over $a$, and we get
	$$\pi(x; M_1/\Q, e) \le \frac{2}{1-\ln q/\ln x}\cdot \frac{x}{4 \ln x} = \frac{2}{1- 1/2\Delta(\ell,2)\eta} \cdot \frac{x}{4 \ln x}. $$
	So we get a positive density $C$ of primes that are inert in $M_1/\Q$
	\begin{equation}
		\pi(x; M_1/\Q, \hat{e}) \ge (1-\frac{1}{2- 1/\Delta(\ell,2)\eta} -\epsilon) \frac{x}{\ln x}  = C\frac{x}{\ln x},
	\end{equation}
	when $x \ge x_0(\epsilon, \epsilon_0)$ is sufficiently large. Primes that are inertia in $M_1$ must be split in exactly two of $K_j$ in $\{ K_m, K'_1, K'_2, K'_3 \}$. Therefore by pigeon hole principle, there exist at least two such $K_j$ satisfy 
	\begin{equation}
	\pi(x; K_j, e) \ge \frac{C}{{4\choose 2}} \cdot \frac{x}{\ln x}\ge c\frac{x}{\ln x},
	\end{equation}
	which implies that $K_j$ is $\theta_j$-good. The second inequality comes from $\eta > \eta_0(1+\epsilon_0)$ and the assumption on $c$. Then by Lemma \ref{lem:EVW}, we get
		$$|\Cl_{K_j}[\ell]| \le O_{\epsilon,\epsilon_0}(\Disc(K_j)^{1/2-\theta_j+\epsilon}).$$
	By Lemma \ref{lem:disc-prod} and \ref{lem:class-grp-decomposition} and Lemma \ref{lem:rank3-disc-bound}, we have for every $L$ that
		\begin{equation}
		\begin{aligned}
		|\Cl_L[\ell]|  \le O_{\epsilon,\epsilon_0}(\Disc(L)^{1/2- 2\Delta(\ell,2)\eta/(4\eta+2)+\epsilon }).\\
		\end{aligned}
		\end{equation}
	So we prove this theorem with 
		$$\delta'_{ic}(\eta, \ell) = \frac{\Delta(\ell,2)\eta}{(2\eta+1)}.$$
\end{proof}

\subsection{Savings for Even $A$ with Rank $3$}\label{ssec:even-rank3-results}
Finally combining Theorem \ref{thm:even-rank3-comparable} and \ref{thm:even-rank3-incomparable}, we get the following theorem. 
\begin{theorem}\label{thm:even-rank3-results-overQ}
		Given $A =( \mathbb{Z}/2\mathbb{Z})^3$ and an odd prime integer $\ell$. For any $A$-extension $L/\Q$, we have
		$$ |\Cl_L[\ell]| \le O_{\epsilon}(\Disc(L)^{1/2 - \delta +\epsilon})$$
		for some 
		$$\delta = \delta'_c(\eta_0, \ell) = \frac{\Delta(\ell,4)}{4\eta_0+2},$$
		where $\eta_0 = \frac{1}{\Delta(\ell,2)}$. 
\end{theorem}
\begin{proof}
	Similarly with Theorem \ref{thm:odd-rank2-final-Q} we can take $\epsilon_0$ arbitrarily small. Notice that $\delta'_c(\eta, \ell)$ decreases as $\eta$ increases and $\delta'_{ic}(\eta, \ell)$ increases as $\eta$ increases. We compare
	$$\delta'_c(\eta_0, \ell)= \frac{1}{48\ell^2+12\ell}, \quad\quad \delta'_{ic}(\eta_0, \ell)= \frac{1}{4\ell+1}.$$
	So the worst point in all range of $\eta$ is the exactly at $\eta= \eta_0$. We can pick $\delta = \frac{\Delta(\ell,4)}{4\eta_0 +2} =  \frac{1}{48\ell^2+12\ell}$.
\end{proof}
\begin{remark}\label{rmk:even-rank3-compare}
	Comparing the saving we get in Theorem \ref{thm:even-rank2-results-overk} and \ref{thm:even-rank3-results-overQ}, here we get an improvement over $\Q$, i.e., 
	$$ \frac{1}{48\ell^2+12\ell} > \frac{1}{64\ell^2+4\ell}$$
	for arbitrary $\ell>1$. 
\end{remark}

\subsection{Induction}\label{ssec:even-induction}
In this section, we will derive $\ell$-torsion bound for every $A = (\mathbb{Z}/2\mathbb{Z})^r$ with $r>2$. Following the Remark \ref{rmk:even-rank3-compare}, we will use Theorem \ref{thm:even-rank3-results-overQ} to prove a point-wise saving for elementary $2$-abelian group with rank greater than $3$. 
\begin{theorem}[Even Exponent, Over $\Q$]\label{thm:even-induction-overQ}
	Given $A = (\mathbb{Z}/2\mathbb{Z})^r$ with $r>2$ and an arbitrary integer $\ell = \ell_{(2)} \ell_{2}>1$. For any $A$-extension $L/\Q$, we have the pointwise bound
	$$ |\Cl_L[\ell]| \le O_{\epsilon}(\Disc(L/k)^{1/2 - \delta(\ell_{(2)})+\epsilon}),$$
	for $\delta(\ell) = \frac{1}{48\ell^2+12\ell}$. 
\end{theorem}
\begin{proof}
	By a similar proof of Theorem \ref{thm:odd-final-results-overk},
	\begin{equation}
	|\Cl_L[\ell]| = \prod_s |\Cl_{F_s}[\ell]|^{1/7} \le O_{\epsilon}(\prod_s \Disc(F_s)^{1/2-\delta+\epsilon})^{1/7} \le O_{\epsilon}(\Disc(L)^{1/2-\delta+\epsilon}).
	\end{equation}	
	where $F_s$ ranges over all degree $8$ subfields of $L$. It follows directly from Corollary \ref{coro:class-grp-second-layer-decomposition} and (\ref{eqn:disc-prod-second-layer}). Similarly with Theorem \ref{thm:odd-final-results-overk}, we derive the results for general $\ell$ by $|\Cl_L[\ell]| =|\Cl_L[\ell_{(2)}]|\cdot| \Cl_L[\ell_2]|$.
\end{proof}

\begin{remark}[Even Exponent, $\ell=3$, Over $\Q$]\label{rmk:even-3torsion}
	When $\ell=3$, we can do induction over an even better result from \cite{EV07} that $|\Cl_F[3]|\le O(\Disc(F)^{1/3+\epsilon})$ for any quadratic extension $F/\Q$. From a direct use of Corollary \ref{coro:class-grp-second-layer-decomposition} and (\ref{eqn:disc-prod-second-layer}), we can take $\delta(3) = 1/3$. 
\end{remark}

When $k\neq \Q$, we use the induction from $r=2$. It follows from a similar proof with Theorem \ref{thm:odd-final-results-overk} directly:
\begin{theorem}[Even Exponent, Over $k$]\label{thm:even-induction-overk}
	Given $A = (\mathbb{Z}/2\mathbb{Z})^r$ with $r\ge 2$ and an integer $\ell>1$. For any $A$-extension $L/k$, we have the pointwise bound
		$$ |\Cl_L[\ell]| \le O_{\epsilon, k}( \Disc(L/k)^{1/2 - \delta_k(\ell_{(2)})+\epsilon}),$$
		for $\delta_k(\ell) = \delta_k(\ell, 2)$ in Theorem \ref{thm:even-rank2-results-overk}. 
\end{theorem}

\section{Acknowledgement}
The author is supported by Foerster-Bernstein Fellowship at Duke University. I would like to thank J\"urgen Kl\"uners, Weitong Wang and Asif Zaman for providing helpful references. I would like to thank Dimitris Koukoulopoulos, Robert J. Lemke Oliver, Melanie Matchett Wood, Asif Zaman and Ruixiang Zhang for helpful conversations. I would like to thank Jordan Ellenberg, Melanie Matchett Wood and Yongqiang Zhao for suggestions on an earlier draft. 


\newcommand{\etalchar}[1]{$^{#1}$}

\Addresses
\end{document}